\documentclass[a4paper, 12pt]{amsart}

\usepackage{amsmath,amsfonts,amssymb,amscd}
\usepackage{verbatim}
\usepackage{stmaryrd}
\input{xy}
\usepackage{enumitem}
\usepackage{hyperref,mdwlist}
\usepackage[english]{babel}
\usepackage{latexsym}
\usepackage{amsopn}
\usepackage[all]{xy}
\usepackage{verbatim}
\usepackage{mathrsfs}
\usepackage{float}
\usepackage{fullpage}
\usepackage{tikz}

\usepackage{graphicx}
\usepackage{tikz-cd}

\newcommand{\XXX}{\mathcal{X}}
\newcommand{\LL}{\mathcal{L}}

\newcommand{\CC}{\mathbb{C}}

\newcommand{\Q}{\mathcal{Q}}
\newcommand{\Ext}{\mathrm{Ext}}
\newcommand{\Hom}{\mathrm{Hom}}

\newcommand{\Gr}{\mathrm{Gr}}

\newcommand{\RR}{\mathbb{R}}
\newcommand{\BBB}{\mathbb{B}}
\newcommand{\Sp}{\mathbb{S}}

\newcommand{\ZZ}{\mathbb{Z}}
\newcommand{\SSS}{\mathcal{S}}

\newcommand{\XX}{\mathcal{X}}

\newcommand{\GG}{\mathrm{G}}

\newcommand{\dd}{\mathbf{d}}
\newcommand{\ff}{\mathbf{f}}
\newcommand{\ee}{\mathbf{e}}
\newcommand{\Gg}{\mathbf{g}}

\newtheorem{thm}{Theorem}[section]
\newtheorem*{Theorem*}{Theorem}
\newtheorem*{Corollary*}{Corollary}

\newtheorem{lem}[thm]{Lemma}
\newtheorem{prop}[thm]{Proposition}
\newtheorem{prop/def}[thm]{Proposition/Definition}
\newtheorem{cor}[thm]{Corollary}

\theoremstyle{definition}
\newtheorem{defn}[thm]{Definition}
\theoremstyle{definition}
\newtheorem{defn/thm}[thm]{Theorem/Definition}

\theoremstyle{remark}
\newtheorem{rmk}[thm]{Remark}
\newtheorem*{rmk*}{Remark}

\title{Specialization map for quiver Grassmannians}
\begin{document}
\address{Giovanni Cerulli Irelli: \newline 
Sapienza-Universit\`a di Roma, 
Via Scarpa 10, 00161 Roma, Italy}
\email{giovanni.cerulliirelli@uniroma1.it}
\address{Francesco Esposito:\newline
University of Padova, via Trieste 63, 35121 Padova, Italy}
\email{esposito@math.unipd.it}
\author[Cerulli Irelli, Esposito, Fang, and Fourier]{Giovanni Cerulli Irelli, Francesco Esposito, Xin Fang, Ghislain Fourier}
\address{Xin Fang:\newline
RWTH Aachen University, Chair of Algebra and Representation Theory, Pontdriesch 10--16, 52062 Aachen, Germany}
\email{xinfang.math@gmail.com}
\address{Ghislain Fourier:\newline
RWTH Aachen University, Chair of Algebra and Representation Theory, Pontdriesch 10--16, 52062 Aachen, Germany}
\email{fourier@art.rwth-aachen.de}

\begin{abstract}
We define a specialization map between cohomology algebras of quiver Grassmannians of Dynkin type and we prove that it is surjective in type A. This generalizes a result of Lanini and Strickland. 
\end{abstract}

\maketitle

\section{Introduction}
Quiver Grassmannians are natural generalizations of Grassmannians and flag varieties. Since every projective variety can be realized as a quiver Grassmannian \cite{ReEvery}, it is natural to restrict the attention to those having representation-theoretic applications, e.g., to canonical bases of quantum groups, dual canonical bases, and cluster algebras. In \cite{CEsp}, a study of the {cohomologies} of quiver Grassmannians was initialized by analyzing in detail the case of quiver Grassmannians attached to indecomposable representations of the Kronecker quiver. In particular, it was shown that such quiver Grassmannians admit an affine paving, and their Hilbert polynomials are explicitly determined. Subsequently, various other classes of quiver Grassmannians are investigated (see \cite{RupelWeist}, \cite{LorscheidWeist1}, \cite{LorscheidWeist2}, \cite{LorscheidWeist3}, \cite{LorscheidWeist4}, \cite{LaPu1}, \cite{LaPu2}) and some general results are obtained in \cite{CEFR}; namely, all quiver Grassmannians attached to complex representations of Dynkin quivers admit affine pavings, thus their {cohomologies are} concentrated in even degrees and are generated by classes attached to algebraic cycles. Furthermore, such properties have also been proved to hold in almost all cases for quivers of affine type, and the algebraicity of the cohomology holds in greater generality, i.e. for quiver Grassmannians attached to rigid representations of acyclic quivers.

Quiver Grassmannians are fibers of a projective family over an affine space whose total space is smooth. This family is called a universal quiver Grassmannian in \cite{CFR, CFFFR}; it depends on two discrete invariants which are the dimension vector of the representation and of the subrepresentation. Moreover universal quiver Grassmannians are endowed with a natural action of a reductive group. Therefore, we may study the variation of the cohomology of the quiver Grassmannians over different orbits {(see \cite{FR} for results in this direction)}. 

The variation of the cohomology of the fibers of a proper map is encoded in the higher direct images of the constant sheaf over the total space. In the case of quiver Grassmannians of Dynkin type these are constructible sheaves over a variety of quiver representations of a given dimension vector, which are constructible with respect to the stratification given by orbits (there are finitely many of them in the Dynkin case). There is an equivalence between constructible sheaves on a given stratified space and functors from its category of \emph{exit paths} \cite[Theorem~1.2]{Tr}.  A great deal of information about these constructible sheaves is contained in  the specialization maps, which are morphisms from the cohomology of a fiber over a point $x_0$ to the cohomology of the fiber over a point in an orbit closing to $x_0$. In significant instances of projective families, e.g. Springer fibers in type $A$ (see \cite{Spalt75}, \cite{DCLP}), one observes the surjectivity of the specialization maps, which indicates a strong form of upper semicontinuity of the cohomology of the fibers of the family in question. Other important instances of surjectivity in cohomology of varieties associated with quivers, arise for the Kirwan map for smooth Nakajima quiver varieties \cite{McGertyNevins}.

A first result on the surjectivity of a specialization map on quiver Grassmannians of Dynkin type has been established in  \cite{LaSt}.
This is part of the framework of linear degenerations of flag varieties of type $A$ introduced and studied in \cite{CFFFR} and \cite{CFFFR2}. In \cite{FR} the authors  determine the supports of the decomposition theorem for the linear degenerations of the complete flag variety by using quantum groups computations. 
Linear degenerations of flag varieties of type $A$ are particular quiver Grassmannians for an equioriented quiver of type $A$. It is hence natural to ask if a specialization map is defined for more general quiver Grassmannians and if it is surjective. 

For a quiver representation $X$ we denote by $\Gr_\ee(X)$ the quiver Grassmannian which parametrizes the subrepresentations of $X$ of dimension vector $\ee$. 

Our main results can be summarized as follows:
\begin{thm}\label{Thm:MainIntro}
Let $\Q$ be a Dynkin quiver, let $\mathbf{d,e}\in \mathbb{N}^{\mathcal{Q}_0}$ be dimension vectors and let $M$ and $N$ be $\Q$-representations of the same dimension vector $\mathbf{d}$. 
\begin{enumerate}
\item[(i)] If $M\leq_{\textrm{deg}}N$ (see Section \ref{Sec:VoR} for the definition of $\leq_{\textrm{deg}}$) then any choice of an exit path  from $N$ to $M$ gives rise to the same specialization map  
$$c_{N,M}:H^\bullet(\Gr_\ee(N))\rightarrow H^\bullet(\Gr_\ee(M)).$$
\item[(ii)] If $\Q$ is of type $A$ then $c_{N,M}$ is surjective.
\end{enumerate}
\end{thm}

Unusually, the restriction to Dynkin quivers in Theorem~\ref{Thm:MainIntro} is not due to the fact that they admit a finite number of orbits, but on the fact that orbit closures of Dynkin quivers are unibranch \cite{Zwara-unibranch}. In particular, the proof of part (i)  generalizes to all representation varieties of algebras where this is known. 

For simplicity, in the introduction we formulate Theorem~\ref{Thm:MainIntro} using exit paths but  we do not explicitly use them in this paper. Instead, we define a specialization map concretely using neighborhoods of the points (see Theorem/Definition~\ref{Def:Specialization}). This corresponds exactly to the construction of  \cite[Theorem~1.2]{Tr}.

The proof of the surjectivity of specialization is obtained by reducing to minimal degenerations, then using a result of Bongartz on minimal degenerations and applying a suitably generalized version of the reduction theorem of \cite[Theorem~24]{CEFR}, to obtain an $\alpha$-partition for which the specialization maps of the single parts are readily determined. This procedure thus yields a comparison theorem between the cohomologies of the quiver Grassmannians of a pair of representations relative to a minimal degeneration. We prove that such a specialization map is surjective and we describe the associated graded of its kernel with respect to a suitable filtration. 
\begin{cor}\label{Cor:MainSemisimple}
Let $\Q$ be a quiver of type $A_n$,  let $\mathbf{e},\mathbf{d}\in\ZZ_{\geq0}^n$ be two dimension vectors, let $M$ be a $\Q$-representation of dimension vector $\mathbf{d}$  and let $\mathbf{0}_\mathbf{d}$ denote the semisimple representation of dimension vector $\mathbf{d}$. Then the inclusion $\Gr_\ee(M)\subset \Gr_\ee(\mathbf{0}_\mathbf{d})=\prod_{i=1}^n\Gr_{\ee_i}(\CC^{\mathbf{d}_i})$ induces a canonical surjective morphism of graded algebras 
$$
\xymatrix{
c_M:\bigotimes_{i=1}^n H^\bullet(\Gr_{\ee_i}(\CC^{\mathbf{d}_i}))
\ar@{->>}[r]&H^\bullet(\Gr_\ee(M)).}
$$
Moreover, $c_{M}=c_{\mathbf{0}_\mathbf{d},M}$.
\end{cor}
We remark that  $c_M$ is uniquely determined by an ideal $I_M$ of $H^\bullet(\Gr_{\ee}(\mathbf{0}_\mathbf{d})$). It is an interesting problem to describe $I_M$ (and hence $H^\bullet(\Gr_\ee(M))$) by generators and relations. In case $M$ is rigid, Corollary~\ref{Cor:MainSemisimple} holds for every quiver \cite[Corollary~43]{CEFR}. 

We now describe the contents of the paper.
In Section~\ref{Sec:Generalities} we collect the tools from representation theory of Dynkin quivers that we need. In Section~\ref{Sec:Specialization} we introduce a general setting in which the specialization map is canonically defined, i.e. independent of the choice of points and paths; moreover we define $\alpha$-partitions and affine pavings of proper $\CC^\times$-equivariant one-dimensional families and prove Proposition~\ref{Prop:OneDimFamily} which is one of the main ingredients in the proof of surjectivity of specialization (Theorem~\ref{Thm:MainIntro}(ii)). In Section~\ref{Sec:SpecQuivGrass} we discuss the specialization map for a universal quiver Grassmannian of Dynkin type;  we show that it satisfies the general setting of Section~\ref{Sec:SpecQuivGrass} (see Lemma~\ref{Lem:GeomSettingQuivGrass}), proving the first two parts of Theorem~\ref{Thm:MainIntro}; in Section~\ref{Subsec:FamilySES} we define one dimensional families associated with non-split short exact sequences and show that they are proper and $\CC^\times$-equivariant; in Section~\ref{subsec:ALfaPartOneDimFam} we show that the family associated with a generating non-split short exact sequence admits an $\alpha$-partition (see Definition~\ref{Def:CellDecFamily}) and in type $A$ it can be refined to an affine paving (Proposition~\ref{Lem:5.6.1/2}); in Section~\ref{Subsec:ProofMainResult} we collect all arguments to prove Theorem~\ref{Thm:MainIntro}. It is worth noticing that the proof of surjectivity of the specialization map reduces to showing that the one-dimensional families associated to minimal degenerations admit an affine paving, or an $\alpha$-partition whose parts have no odd cohomology. We expect that this strategy can be fruitfully put to use in more general contexts (e.g. Dynkin quivers other that of type $A$).
\newline 
\newline
\textbf{Acknowledgements}: 
The first author acknowledges support of the University La Sapienza of Rome (Bando di ateneo RP120172B7D7C329) and of the project funded by NextGenerationEU under NRRP, Call PRIN 2022  No.~104 of February 2, 2022 of Italian Ministry of University and Research; Project 2022S97PMY \textit{Structures for Quivers, Algebras and Representations (SQUARE)}. The second author acknowledges the support of the project of the University of Padova BIRD203834/20 and of the project funded by the European Union -NextGenerationEU under the National Recovery and Resilience Plan (NRRP), Mission 4 Component 2 Investment 1.1 - Call PRIN 2022 No. 104 of February 2, 2022 of Italian Ministry of University and Research; Project 2022S8SSW2 (subject area: PE - Physical Sciences and Engineering) ``Algebraic and geometric aspects of Lie Theory''.
The fourth author gratefully acknowledges support by the Deutsche Forschungsgemeinschaft – Project-ID
286237555– TRR 195.

\section{Generalities}\label{Sec:Generalities}

{ In this paper, if not mentioned otherwise, we fix $\mathbb{C}$, the field of complex numbers, as base field.}

\subsection{Representations of quivers}
Let $\mathcal{Q} = (\mathcal{Q}_0 , \mathcal{Q}_1,s,t)$ be a quiver, that is, an oriented graph; its set of vertices is $\mathcal{Q}_0$ and its set of edges is $\mathcal{Q}_1$. We assume furthermore that the quiver $\mathcal{Q}$ is finite, i.e. both $\mathcal{Q}_0$ and $\mathcal{Q}_1$ are finite sets. A {finite-dimensional} representation of $\mathcal{Q}$ is the data $((V_x)_{x\in \mathcal{Q}_0}, (f_\alpha)_{\alpha\in\mathcal{Q}_1})$, where the $V_x$ are {finite-dimensional} vector spaces, and $f_\alpha: V_{s(\alpha)}\rightarrow V_{t(\alpha)}$ are linear maps between the vector spaces associated with the start $s(\alpha)$ and the terminal $t(\alpha)$ vertices of the arrow $\alpha$.
    
    Let $\mathbf{Rep}(\mathcal{Q})$ be the category of {finite-dimensional} representations of the quiver $\mathcal{Q}$. It is an abelian,  Krull-Schmidt and hereditary category. Given two $\Q$-representations $M$ and $N$ we denote $[M,N]=\dim \Hom_\Q(M,N)$ and $[M,N]^1=\dim \Ext^1_\Q(M,N)$.
    
    Let $V=((V_x)_{x\in \mathcal{Q}_0}, (f_\alpha)_{\alpha\in\mathcal{Q}_1})$ be in $\mathbf{Rep}(\mathcal{Q})$. The vector $\mathbf{dim}\,V=(\dim(V_x))_{x\in\mathcal{Q}_0}\in \mathbb{N}^{\mathcal{Q}_0}$ is called the dimension vector of $V$.
    
    A subrepresentation of $V$ is the data of a subspace $W_x \subseteq V_x$ for every vertex $x$ of $\mathcal{Q}$, such that, for every arrow $\alpha\in\mathcal{Q}_1$, one has $f_\alpha (W_{s(\alpha)})\subseteq W_{t(\alpha)}$.
    
    A quiver is of Dynkin type if its underlying {unoriented} graph is a disconnected union of simply laced Dynkin diagrams. {According to Gabriel's theorem, $\mathbf{Rep}(\mathcal{Q})$ has finitely many isoclasses of indecomposables if and only if $\mathcal{Q}$ is of Dynkin type with finitely many vertices.}
    
    \subsection{Varieties of representations}\label{Sec:VoR}
    Let $\mathbf{d}\in \mathbb{N}^{\mathcal{Q}_0}$ be a dimension vector for the quiver $\mathcal{Q}$ and let $\mathbb{C}^{\mathbf{d}}=\bigoplus_{i\in\Q_0}\CC^{\dd_i}$ be the {$\mathcal{Q}_0$-}graded vector space of dimension vector $\mathbf{d}$. The automorphism group of the graded vector space $\mathbb{C}^{\mathbf{d}}$ is denoted $\mathrm{G}_{\mathbf{d}}=\mathrm{GL}_{\mathbf{d}}(\mathbb{C})=\prod_{i\in\Q_0}\mathrm{GL}_{\mathbf{d}_i}(\CC)$. We denote by $\mathrm{R}_{\mathbf{d}}=\mathrm{R}_{\mathbf{d}}(\Q)=\prod_{\alpha\in\Q_1}\CC^{t(\alpha)\times s(\alpha)}$ the variety of representations of $\Q$ of dimension vector $\mathbf{d}$ (more precisely the variety of $\Q$-representations having $\mathbb{C}^{\mathbf{d}}$ as underlying vector space).  The reductive group $\mathrm{G}_{\mathbf{d}}$ acts from the left on $\mathrm{R}_{\mathbf{d}}$ by change of basis, i.e. $(g_i)\cdot (f_\alpha)=(g_{t(\alpha)} f_\alpha g_{s(\alpha)}^{-1})_{\alpha\in\Q_1}$ and the $\mathrm{G}_{\mathbf{d}}$-orbits correspond to isoclasses. We denote the points of $\mathrm{R}_{\mathbf{d}}$ with lowercase letters but when we consider them as $\Q$-representations we use  upper case letters. Thus we write $m\in \mathrm{R}_{\mathbf{d}}$ or $M\in \mathbf{Rep}(\mathcal{Q})$. We use the notation $[M]$ to denote the isomorphism class of a $\Q$-representation $M$ and we denote by $\mathcal{O}_m$ the $\mathrm{G}_{\mathbf{d}}$-orbit of a point $m\in \mathrm{R}_{\mathbf{d}}$. 

For a Dynkin quiver $\mathcal{Q}$, Gabriel's theorem implies that there are only finitely many isoclasses of representations of $\mathcal{Q}$ of dimension vector $\mathbf{d}$.  The closure of orbits defines a partial order on the set of $\mathrm{G}_{\mathbf{d}}$-orbits in $\mathrm{R}_{\mathbf{d}}$; it is the so-called degeneration order; we denote it $M\leq_{\mathrm{deg}} N$, if 
$\mathrm{G}_\dd \cdot N \subseteq \overline{\mathrm{G}_\dd \cdot M}$, and we say that $N$ is a degeneration of $M$.
    
    \begin{rmk}
    For a Dynkin quiver $\mathcal{Q}$, the degeneration order can be expressed in terms of the category $\mathbf{Rep}(\mathcal{Q})$ as the $\mathrm{Hom}$ order or also as the $\mathrm{Ext}$ order \cite{Bongartz}. This result is not used directly in the present paper; however, it is indirectly used through the characterization of Bongartz of minimal degenerations.
    \end{rmk}

\subsubsection{Minimal degenerations}
We recall the fundamental result of Bongartz on minimal degenerations (formulated in the special case of Dynkin quivers) and derive from it a consequence used in the study of specialization maps.
\begin{thm}[(Bongartz, \cite{Bo}, Theorem 4)]\label{thm:Bongartz}
Let $M\leq_{\mathrm{deg}} N$ be a degeneration of representations of a Dynkin quiver $\mathcal{Q}$. Then, it is minimal if and only if there is an exact sequence 
 $$
0 \rightarrow X_1 \rightarrow Y_1 \rightarrow S_1 \rightarrow 0
 $$
 and a module $Y'$ such that 
 \begin{enumerate}[label=(\roman*)]
     \item $X_1$ and  $S_1$ are indecomposable 
     \item $X_1\oplus S_1$ is a minimal degeneration of $Y_1$.
     \item $M = Y_1 \oplus Y'$;
     \item $N = X_1 \oplus S_1 \oplus Y'$;
     \item for any common indecomposable direct summand $B\not\simeq S_1$ of $N$ and $M$, one has
     $$
     [B, M]^1= [B, N]^1;
     $$
     \item for any common indecomposable direct summand $C\not\simeq X_1$ of $N$ and $M$, one has
     $$
[M, C]^1 = [N, C]^1.
     $$
 \end{enumerate}
\end{thm}
 We now state a consequence of Theorem~\ref{thm:Bongartz} to be used in the study of the specialization map for type $A$ quivers.
\begin{thm}\label{Thm:BongartzTypeA}
  Let $M \leq_{\mathrm{deg}} N$ be a minimal degeneration between representations of a quiver of type $A$. Then, there are
  \begin{enumerate}[label=(\alph*)]
      \item indecomposables $X_1$, $S_1$;
      \item a nonsplit short exact sequence
      $$
      \xi_1 : 0 \rightarrow X_1 \rightarrow Y_1 \rightarrow S_1 \rightarrow 0;
      $$
      \item representations $X' , S'$;
  \end{enumerate}
  satisfying the following conditions:
  \begin{enumerate}[label=(\roman*)]
      \item $M = Y_1 \oplus X' \oplus S'$;
      \item $N = X_1 \oplus S_1 \oplus X' \oplus S'$;
      \item $[S_1 \oplus S' , X_1 \oplus X']^1 = [S_1 , X_1]^1$;
      \item $[X_1 , X']^1 = [S' , S_1]^1= 0$.
  \end{enumerate}
\end{thm}
\begin{proof}
  From Theorem~\ref{thm:Bongartz}, one gets the data 
  $(X_1 , S_1 , \xi_1 , Y')$. Therefore, to define the data $(X_1 , S_1 , \xi_1 , X' , S')$ that meet the requirements of the theorem, one needs to suitably split $Y'$ as a direct sum $X' \oplus S'$.
  
 Since $Y_1$ is the middle term of a non-split short exact sequence with indecomposable extremes, {under the type $A$ assumption,} it is straightforward to check that, for any indecomposable $B$, one has either 
  $\Ext^1 (B, Y_1) = 0$ or $\Ext^1 (Y_1, B) = 0$; furthermore, if $B,C$ are two indecomposables such that 
  $\Ext^1 (C, Y_1) \neq 0$ and $\Ext^1 (Y_1, B) \neq 0$, then it follows that 
  $\Ext^1 (B, C) = 0$. 
  
  Denote by $\Sigma$ the set of isomorphism classes of indecomposable direct summands of $Y'$ that are different from $X_1$ and $S_1$. Then, the set $\Sigma' = \Sigma \cup \{Y_1\}$ may be linearly ordered in such a way that, for $B,C\in \Sigma'$ such that $A \leq B$, one has $\Ext^1 (A, B) = 0$. With respect to this linear order, let $Y' = X' \oplus S'$, where $S'$ is the sum of the indecomposables $B$ of $Y'$ corresponding to an element of $\Sigma$ for which $B \leq  Y_1$ and $X'$
  is the sum of the indecomposables $C$ of $Y'$ corresponding to an element of $\Sigma$ for which 
  $C \geq  Y_1$.
  
  Then, by the property of the ordering, one gets:
  $$
  \Ext^1(S' , X') = \Ext^1(S' , Y_1) = \Ext^1(Y_1 , X') = 0.
  $$
  Furthermore, since Theorem~\ref{thm:Bongartz} implies that for $B$ an indecomposable direct summand of $X' \oplus S'$ different from both $X_1$ and $S_1$, one has 
  $$
  \Ext^1 (B, Y_1) = \Ext^1 (B, X_1 \oplus S_1) \ \ \mathrm{and} \ \ \Ext^1 (Y_1, B) = \Ext^1 (X_1 \oplus S_1, B),
  $$
  it follows that 
  $$
  \Ext^1(S' , X') = \Ext^1(S' , S_1) = \Ext^1(X_1 , X') = 0.
  $$
  This concludes the proof.
\end{proof}
\subsection{Quiver Grassmannians}\label{SubSec:QuivGrass}
    Quiver Grassmannians are schematic fibers of projective families called universal quiver Grassmannians (see \cite{Schofield}, \cite{CFR}) defined as follows: 
      one considers  the incidence variety 
    $$
    \XXX_{\mathbf{e}, \mathbf{d}} = \Gr_\mathbf{e}^Q(\mathbf{d}):=\{((U_i), (M_\alpha))\in \Gr_\mathbf{e}(\CC^\mathbf{d})\times \mathrm{R}_\mathbf{d}\ \ | \ \ M_\alpha(U_{s(\alpha)})\subseteq U_{t(\alpha)}\;\forall \alpha\in Q_1\}
    $$
 equipped with two projections 
        $$
    \xymatrix{
    \Gr_\mathbf{e}(\CC^\mathbf{d})& 
    \XXX_{\mathbf{e}, \mathbf{d}}\ar_{\ \ p_1}[l]\ar^{p_2}[r]&\mathrm{R}_\mathbf{d}
    }
    $$
    which are $\mathrm{G}_{\mathbf{d}}$-equivariant. The map $p_1$ realizes $\XXX_{\mathbf{e}, \mathbf{d}}$ as the total space of a homogeneous vector bundle on $\Gr_\mathbf{e}(\CC^\mathbf{d})$ (see \cite[section 2.2]{CFR}). This implies that it is irreducible and smooth.
      
      The map  $\pi_{\mathbf{e}, \mathbf{d}} = p_2$ is proper, and its image is the closed $\mathrm{G}_{\mathbf{d}}$-stable subvariety of $\mathrm{R}_\mathbf{d}$ consisting of those representations admitting a subrepresentation of dimension vector $\mathbf{e}$.
     The morphism $\pi_{\mathbf{e}, \mathbf{d}}$ is the universal family of quiver Grassmannians of $\mathbf{e}$-dimensional subrepresentations of $\mathbf{d}$-dimensional representations. For $M\in\mathrm{R}_\mathbf{d}$, the quiver Grassmannian ${\Gr_\mathbf{e}(M)}$ is the fiber $\pi_{\mathbf{e}, \mathbf{d}}^{-1}(M)$ of $\pi_{\mathbf{e}, \mathbf{d}}$ over $M$.

\subsection{Generating short exact sequences}\label{Sec:generating}
Let $\Q$ be a quiver and let 
$$
\xymatrix{
\xi:0\ar[r]&X\ar^\iota[r]&Y\ar^p[r]& S\ar[r]&0
}
$$
be a short exact sequence. Following Caldero and Chapoton \cite{CC}, for every dimension vector $\mathbf{e}$ there is a map $\Psi:\Gr_\mathbf{e}(Y)\rightarrow \coprod_{\ff+\Gg=\ee}\Gr_\ff(X)\times\Gr_\Gg(S)$ given by $N\mapsto (\iota^{-1}(N),p(N))$. This is not a nice map in general, but as noticed in \cite{CEFR}, it becomes very useful if $\xi$ has good homological properties. 
Following \cite{CEFR}, we say that $\xi$ is \emph{generating} if $\Ext^1(S,X)=\CC\xi$. Thus, $\xi$ is either split and $[S,X]^1=0$ or $\xi$ is not split and $[S,X]^1=1$. For every $\ff+\Gg=\ee$ we denote $\mathcal{S}_{\ff,\Gg}^\xi=\Psi^{-1}(\Gr_\ff(X)\times\Gr_\Gg(S))$ and 
$\Psi_{\ff,\Gg}^\xi:{\mathcal{S}_{\ff,\Gg}^\xi}\rightarrow \Gr_\ff(X)\times\Gr_\Gg(S)$. We recall from \cite[Section~1.3]{DCLP} that an $\alpha$-partition of an affine variety $X$ is a finite partition $X=X_1\cup\cdots\cup X_n$ into locally closed subsets such that $X_1\cup\cdots\cup X_i$ is closed for every $i=1,\cdots, n$. It is proved in 
\cite[Lemma 20]{CEFR} that the collection of the strata $\mathcal{S}_{\ff,\Gg}^\xi$ forms an $\alpha$-partition of $\Gr_\mathbf{e}(Y)$. 
It is now important to describe the image of $\Psi_{\ff,\Gg}^\xi$. If $\xi$ splits then $\mathrm{Im}(\Psi_{\ff,\Gg}^\xi)=\Gr_\ff(X)\times\Gr_\Gg(S)$. If $\xi$ is not split, then one defines 
\begin{equation}\label{Eq:DefXsSX}
\begin{array}{cc}
X_S = \mathrm{Ker}(X\rightarrow \tau S)\subset X,& S^X = \mathrm{Im}(\tau^{-}X\rightarrow S)\subset S,
\end{array}
\end{equation}
where $\tau$ denotes the AR-translate and $\tau^-$ its quasi-inverse. By \cite[Theorem~31]{CEFR}, 
\begin{eqnarray}
\mathrm{Im}(\Psi_{\ff,\Gg}^\xi)&=&\{(N_1,N_2)|\, [N_2,X/N_1]^1=0\}\nonumber\\
&=&\left(\Gr_\ff(X)\times\Gr_\Gg(S)\right)\setminus\left(\Gr_\ff(X_S)\times \Gr_{\Gg-\mathbf{dim}S^X}(S/S^X)\right).
\end{eqnarray}
Moreover $\Psi_{\ff,\Gg}^\xi$ is a (locally trivial) affine bundle over its image. 

We can now interpret Theorem~\ref{Thm:BongartzTypeA} in this language and notice that for quivers of type $A$, every minimal degeneration is given by a generating short exact sequence. 
\begin{cor}\label{Cor:MinDegGenerating}
Let $\Q$ be a quiver of type $A$ and let $M\leq_{\mathrm{deg}}N$ be a minimal degeneration. Then, {in} the notation of Theorem~\ref{Thm:BongartzTypeA}, the short exact sequence 
$$
\xymatrix{
\xi:0\ar[r]&X=X_1\oplus X'\ar[r]&Y=M\ar[r]&S=S_1\oplus S'\ar[r]&0
}
$$
is generating. Moreover, $X_S=X'\oplus (X_1)_{S_1}$ and $S^X=(S_1)^{X_1}$. 
\end{cor}
\begin{proof}
It follows from Theorem~\ref{Thm:BongartzTypeA}.
\end{proof}    

\section{Specialization}\label{Sec:Specialization}

We consider a geometric setting given by a datum $(\GG,X,Y,\pi)$ satisfying the following: 
\begin{eqnarray}
&&G\textrm{ is a connected algebraic group};\label{Def:SpecializationG}\\ 
&&X \textrm{ and }Y \textrm{ are }G\textrm{-varieties};\label{Def:SpecializationGVar}\\
&&Y \textrm{ has finitely many }G\textrm{-orbits};\label{Def:SpecializationFiniteGOrbits}\\
&& \textrm{the closure }\overline{\mathcal{O}}\textrm{ of each }G\textrm{-orbit }\mathcal{O}\subset Y\textrm{ is unibranch, i.e. the normalization map }\label{Def:SpecializationUnibranch}\\
&&\nu: \widetilde{\overline{\mathcal{O}}}\rightarrow \overline{\mathcal{O}}\textrm{ is a bijection};\nonumber\\
&&\textrm{the stabilizer of each point of }Y\textrm{ is connected};\label{Def:SpecializationStab}\\
&&\pi:X\rightarrow Y \textrm{ is a proper and G-equivariant map}. \label{Def:SpecializationProper}
\end{eqnarray}
In this section we define a map, that we call a specialization map, 
$$
c_{y_1,y_2}:H^\bullet(\pi^{-1}(y_1))\rightarrow H^\bullet(\pi^{-1}(y_2))
$$
for any points $y_1,y_2\in Y$ such that $y_1\in\overline{G\cdot y_2}$ which depends only on the $G$-orbits of $y_1$ and $y_2$. 

For a complex algebraic variety $Z$ we denote by $H^\bullet(Z):=H^\bullet(Z,\ZZ)$ the singular  cohomology of $Z$ with integer coefficients.

\begin{lem}\label{Lem:Transitivity}
Let $y_1$ and $y_2$ be two points in the same $G$-orbit $\mathcal{O}$ in $Y$. Then there is a canonical isomorphism $\varphi_{y_1,y_2}:H^\bullet(\pi^{-1}(y_1))\rightarrow H^\bullet(\pi^{-1}(y_2))$ between cohomology algebras.
In particular, $\varphi_{y_2,y_3}\circ \varphi_{y_1,y_2}=\varphi_{y_1,y_3}$ and $\varphi_{y_1,y_1}=Id$ for every $y_1,y_2,y_3\in \mathcal{O}$. 
\end{lem}
\begin{proof}
Let us define $\varphi_{y_1,y_2}$. Let $g\in\GG$ be such that $g. y_2=y_1$. This defines an isomorphism  $\xymatrix{\phi_{y_2,y_1,g}:\pi^{-1}(y_2)\ar^(.6){\simeq}[r]& \pi^{-1}(y_1)}$ of proper varieties (by \eqref{Def:SpecializationProper}). Since cohomology is a contravariant functor one gets an isomorphism
$$
\varphi_{y_1,y_2,g}:\xymatrix{H^\bullet(\pi^{-1}(y_1))\ar^{\simeq}[r]&H^\bullet(\pi^{-1}(y_2))}.
$$ 
Clearly, if $g.y_2=y_1$ and $h.y_3=y_2$ then 
\begin{equation}\label{Eq:phiLemma}
\phi_{y_2,y_1,g}\circ \phi_{y_3,y_2,h}=\phi_{y_3,y_1,gh}.
\end{equation}
In cohomology this yields
\begin{equation}\label{Eq:VarphiLemma}
\varphi_{y_2,y_3,h}\circ \varphi_{y_1,y_2,g} =\varphi_{y_1,y_3,gh}.
\end{equation}

We need to show that $\varphi_{y_1,y_2,g}$ is independent of $g$.  By \eqref{Eq:VarphiLemma}, $\varphi_{y_2,y_1,g^{-1}}=\varphi_{y_1,y_2,g}^{-1}$. 
Let $h\in\GG$ be another element such that $h.y_2=y_1$. Then $\varphi_{y_1,y_2,h}^{-1}\circ \varphi_{y_1,y_2,g}=\varphi_{y_1,y_1,gh^{-1}}$. We observe that $gh^{-1}\in \textrm{Stab}_\GG(y_1)$ which is  connected by \eqref{Def:SpecializationStab}. Since the action of a connected transformation group is trivial on cohomology, it follows that $\varphi_{y_1,y_1,gh^{-1}}=Id$ and thus $\varphi_{y_1,y_2,g}=\varphi_{y_1,y_2,h}$. This proves also transitivity.
\end{proof}
\begin{defn/thm}\label{Def:HSpecialization}
Let $\mathcal{O}\subset Y$ denote a $G$-orbit.  Then  there is a well-defined inverse system of graded rings $\mathcal{Z}_\mathcal{O}=\{(\varphi_{y_1,y_2}:H^\bullet(\pi^{-1}(y_1))\rightarrow H^\bullet(\pi^{-1}(y_2)))|\, y_1,y_2\in \mathcal{O}\}$. We define 
$$
H^\bullet(X_{[\mathcal{O}]})=\varprojlim_{\mathcal{Z}_\mathcal{O}} \varphi_{y_1,y_2}.
$$
\end{defn/thm}
\begin{proof}
It follows from Lemma~\ref{Lem:Transitivity}.  
\end{proof}
\begin{rmk}
We stress the fact that $H^\bullet(X_{[\mathcal{O}]})$ is not the cohomology of a geometric object, but it  is canonically isomorphic to $H^\bullet(\pi^{-1}(y))$ for any $y\in\mathcal{O}$. We think of  $X_{[\mathcal{O}]}$ as representing every fiber of $\pi$ over a point in $\mathcal{O}$. 
 \end{rmk}
To prove our next result we need the following basic lemma which was communicated to us by A. Rapagnetta. 
\begin{lem}\label{lem:Rap}
Let $Z$ be a normal variety. Let $z\in Z$ and let $D$ be a connected open neighborhood of $z$. Let $D_s$ be a Zariski open subset of the smooth locus of $D$. Then $D_s$ is connected.
\end{lem}
\begin{proof}
Let $f:\tilde{Z}\rightarrow Z$ be a resolution of singularities of $Z$. Let $\tilde{Z}_D:=f^{-1}(D)$. Since $Z$ is normal, $f_\ast(\mathcal{O}_{\tilde{Z}})=\mathcal{O}_{Z}$ and by restriction to the open subset $\tilde{Z}_D$, also $f_\ast(\mathcal{O}_{\tilde{Z}_D})=\mathcal{O}_D$. In particular, $H^0(\tilde{Z}_D,\mathcal{O}_{\tilde{Z}_D})=H^0(D,\mathcal{O}_D)$. Since $D$ is normal, $H^0(D,\mathcal{O}_D)$ has no zero-divisors. It follows that $H^0(\tilde{Z}_D,\mathcal{O}_{\tilde{Z}_D})$ has no zero-divisors. Thus, since $\tilde{Z}_D$ is smooth, it must be connected. Moreover, $D_s\simeq f^{-1}(D_s)$ is an open subset of the connected variety $\tilde{Z}_D$ which is the complement of subvarieties of lower dimension. Hence, it follows that $D_s$ is connected.
\end{proof}
\begin{defn}
An open neighborhood $U$ of a point $y\in Y$ is called a \emph{nice little neighborhood} of $y$ if it satisfies the following two conditions:
\begin{enumerate}
\item[(i)] $U$ is contractible; 
\item[(ii)] the map in cohomology $\psi: H^\bullet(\pi^{-1}(U))\rightarrow H^\bullet(\pi^{-1}(y))$ induced by the inclusion $\{y\}\subset U$ is an isomorphism.
\end{enumerate}
\end{defn}
\begin{lem}\label{Lem:OpenNeigh}
For every $y\in Y$ the following hold:
\begin{enumerate}
\item[(i)] Nice little neighborhoods of $y$ form a basis of open neighborhoods of $y$.
\item[(ii)] Let $\mathcal{O}$ be a $\GG$-orbit such that $y\in\overline{\mathcal{O}}$. Then for every nice little neighborhood $U$ of $y$, the intersection $U\cap \mathcal{O}$ is connected.
\end{enumerate}
\end{lem}
\begin{proof}
Let us prove $(i)$. Let $V$ be an open subset of $Y$ containing $y$. 
Since $\pi$ is a proper and $\GG$-equivariant map, it follows that $\pi_\ast(\ZZ_{X})_{|_V}$ is a complex of sheaves with cohomology constructible for the stratification of the $\GG$-orbits of $Y$ (\cite[Proposition of section~1.9]{GMP2}). Hence there exists a  contractible   open neighborhood $U$ of  $y$, contained in $V$, such that $H^\bullet(\pi_\ast(\ZZ_{X})_{|U})$ coincides with the stalk $H^\bullet(\pi_\ast(\ZZ_{X})_{y})$ at $y$. Moreover, by \cite[Prop. 2.5.2]{Kashiwara-Schapira}, $$H^\bullet(\pi^{-1}(U))=H^\bullet(\pi_\ast(\ZZ_{X})_{|U})\simeq H^\bullet(\pi^{-1}(y)).$$  Thus every open neighborhood of $y$ contains a nice little neighborhood of $y$.

Let us prove $(ii)$. Let now $\mathcal{\mathcal{O}}\subset Y$ be a $\GG$-orbit such that $y\in \overline{\mathcal{O}}$. Let $Z$ be a normalization of $\overline{\mathcal{O}}$. By \eqref{Def:SpecializationUnibranch}, the normalization map $\nu:Z\rightarrow \overline{\mathcal{O}}$ is a homeomorphism. We put $D=\nu^{-1}(U)$. Since $\nu$ is a homeomorphism, $D$ is an open and connected subset of $Z$. Since $\mathcal{O}\subset \overline{\mathcal{O}}$ is smooth, $\nu$ is an isomorphism over $\mathcal{O}$ and thus $\nu^{-1}(\mathcal{O})$ is an open subset of the smooth locus of $Z$.  By lemma~\ref{lem:Rap}, $D\cap \nu^{-1}(\mathcal{O})$ is connected.  Thus $U\cap\mathcal{O}=\nu(D\cap \nu^{-1}(\mathcal{O}))$ is connected.

\end{proof}
Let $\mathcal{O}_1\subset\overline{\mathcal{O}_2}\subset Y$. Let $y_1\in \mathcal{O}_1$ and let $U$ be a nice little  neighborhood of $y_1$. Let $y_2\in U\cap\mathcal{O}_2$. Then one has a homomorphism of graded algebras $c_{y_1,U,y_2}:H^\bullet(X_{[\mathcal{O}_1]})\rightarrow H^\bullet(X_{[\mathcal{O}_2]})$ defined as the composition 
$$
\xymatrix{
H^\bullet(X_{[\mathcal{O}_1]})\ar_{c_{y_1,U,y_2}}[d]\ar^{\simeq}[r]&H^\bullet(\pi^{-1}(y_1)) \ar^{\simeq}[r]& H^\bullet(\pi^{-1}(U))\ar[d]\\
H^\bullet(X_{[\mathcal{O}_2]})&H^\bullet(\pi^{-1}(y_2))\ar_\simeq[l]& H^\bullet(\pi^{-1}(y_2)) \ar@{=}[l] 
}
$$
where the right vertical arrow is induced by the inclusion $\{y_2\}\subset U$. 
\begin{lem}\label{Lem:KeyLemSpec}
Let $y_1\in \mathcal{O}_1$, let $U$ be a nice little neighborhood of $y_1$ and let $y_2\in U\cap \mathcal{O}_2$. Then  
\begin{enumerate}
\item[(i)] $c_{y_1,U,y_2}=c_{g.y_1,g.U,g.y_2}$ for every $g\in \GG$.
\item[(ii)] $c_{y_1,U,y_2}=c_{y_1,U'',y_2}$ for every nice little neighborhood $U''$ of $y_1$ such that $U''\subseteq U $ and $y_2\in U''$.
\item[(iii)] $c_{y_1,U,y_2}=c_{y_1,U,y_2''}$ for every $y_2''\in U\cap\mathcal{O}_2$.
\end{enumerate}
\end{lem}
\begin{proof}
To prove $(i)$, observe that one has the following diagram
$$
\xymatrix@C=20pt{
H^\bullet(X_{[\mathcal{O}_1]})\ar_{c_{g.y_1,g.U,g.y_2}}[ddd]\ar@{=}[dr]\ar^{\simeq}[rr]\ar@{}|*+[F-:<3pt>]{1}[drr]&&H^\bullet(\pi^{-1}(g.y_1)) \ar^{\varphi_{g.y_1,y_1}}[d]\ar^{\simeq}[rr]\ar@{}|*+[F-:<3pt>]{2}[dr]& &H^\bullet(\pi^{-1}(g.U))\ar_{g^\ast}[dl]\ar^\beta[ddd]\\
&H^\bullet(X_{[\mathcal{O}_1]})\ar_{c_{y_1,U,y_2}}[d]\ar^{\simeq}[r]&H^\bullet(\pi^{-1}(y_1)) \ar^{\simeq}[r]& H^\bullet(\pi^{-1}(U))\ar^\alpha[d]\ar@{}|*+[F:<3pt>]{3}[dr]&\\
&H^\bullet(X_{[\mathcal{O}_2]})\ar@{=}[dl]&H^\bullet(\pi^{-1}(y_2))\ar_\simeq[l]& H^\bullet(\pi^{-1}(y_2)) \ar@{=}[l]
\ar@{}|*+[F:<3pt>]{4}[dl]& \\
H^\bullet(X_{[\mathcal{O}_2]})\ar@{}|*+[F:<3pt>]{5}[urr]&&H^\bullet(\pi^{-1}(g.y_2))\ar_{\varphi_{g.y_2,y_2}}[u]\ar_\simeq[ll]& &H^\bullet(\pi^{-1}(g.y_2)) \ar@{=}[ll] \ar^{\varphi_{g.y_2,y_2}}[ul]}
$$
The squares $\xymatrix{\ar@{}|*+[F-:<3pt>]{2}[]}$, $\xymatrix{\ar@{}|*+[F-:<3pt>]{3}[]}$ and $\xymatrix{\ar@{}|*+[F-:<3pt>]{4}[]}$ commute by functoriality of cohomology and the definition of $\varphi_{g.y_1,y_1}$ and $\varphi_{g.y_2,y_2}$ given in Lemma~\ref{Lem:Transitivity}.  The squares $\xymatrix{\ar@{}|*+[F-:<3pt>]{1}[]}$ and $\xymatrix{\ar@{}|*+[F-:<3pt>]{5}[]}$ commute by the universal property of inverse limits. It follows that the leftmost square is commutative, i.e., $c_{g.y_1,g.U,g.y_2}=c_{y_1,U,y_2}$.

To prove $(ii)$ we consider the following commutative diagram 
$$
\xymatrix{
&H^\bullet(\pi^{-1}(U))\ar[dd]\ar[dr]\ar_\cong[dl]&\\
H^\bullet(\pi^{-1}(y_1))&&H^\bullet(\pi^{-1}(y_2))\\
&H^\bullet(\pi^{-1}(U''))\ar[ur]\ar^\cong[ul]&
}
$$
where all the maps are induced by inclusion. The upper and lower path from left to right thus coincide and are $c_{y_1,U,y_2}$ and $c_{y_1,U'',y_2}$, respectively.

To prove \eqref{Def:SpecializationG} it is enough to prove that the following diagram commutes:
$$
\xymatrix{
H^\bullet(\pi^{-1}(y_2'))\ar^{\varphi_{y_2',y_2}}_\simeq[rr]&&H^\bullet(\pi^{-1}(y_2))\\
&H^\bullet(\pi^{-1}(U))\ar^\varphi[ur]\ar_{\varphi'}[ul]&
}
$$
where $\varphi'$ and $\varphi$ are induced by the inclusions $\{y_2'\}\subset U$ and $\{y_2\}\subset U$, respectively. Let $g\in \GG$ such that $g.y_2=y_2'$. Let us denote the two inclusions $\pi^{-1}(y_2)\subset \pi^{-1}(U)$ and $\pi^{-1}(y_2')\subset \pi^{-1}(U)$ by $\iota$ and $\iota'$, respectively. By Lemma~\ref{Lem:OpenNeigh}, the intersection $\mathcal{O}\cap U$ is connected. Thus, there exists a continuous path $\eta:[0,1]\rightarrow \GG$ such that $\eta(0)=Id$, $\eta(1)=g$ and $\eta(t).y_2\in U$ for every $t\in[0,1]$. Consider the map 
$$
\omega: \pi^{-1}(y_2)\times [0,1]\rightarrow \pi^{-1}(U)
$$ defined by $\omega(x,t)=\eta(t).x\in X$. Thus $\omega(-,0)=\iota$ and $\omega(-,1)=\iota'\circ \phi_{y_2,y_2', g}$. Since $\omega$ is a homotopy between $\omega(-,0)$ and $\omega(-,1)$, both $\omega(-,0)$ and $\omega(-,1)$ induce the same map in cohomology. These induced maps are precisely $\varphi$ and $\varphi_{y_2',y_2}\circ\varphi'$.
\end{proof}

\begin{thm}\label{Thm:SpecializationDef}
Let $\mathcal{O}_1\subset \overline{\mathcal{O}_2}$ be two $G$-orbits in $Y$. Let $y_1,y_1'\in \mathcal{O}_1$, let $y_1\in U$ and $y_1'\in U'$ be nice little neighborhoods and let $y_2\in U\cap \mathcal{O}_2$ and $y_2'\in U'\cap \mathcal{O}_2$. Then $c_{y_1,U,y_2}=c_{y_1',U',y_2'}$.
\end{thm}

\begin{proof}
Let $g\in \GG$ such that $y_1'=g.y_1$. 
By Lemma~\ref{Lem:KeyLemSpec} there is the following chain of equalities
$$c_{y_1,U,y_2}=c_{y_1',g.U,g.y_2}=c_{y_1', g.U, y_2''}=c_{y_1',U'', y_2''}=c_{y_1',U',y_2''}=c_{y_1',U',y_2'}$$
where $U''\subseteq g.U\cap U'$ is a nice little neighborhood of $y_1'$ which exists by Lemma~\ref{Lem:OpenNeigh}(i) and $y_2''\in U''\cap\mathcal{O}_2$.
\end{proof}
\begin{defn/thm}\label{Def:Specialization}
Suppose that the datum $(G,X,Y,\pi)$ satisfies the conditions $\eqref{Def:SpecializationG}-\eqref{Def:SpecializationProper}$. Then for every two $\GG$-orbits $\mathcal{O}_1\subset\overline{\mathcal{O}_2}\subset Y$ there is a canonical specialization map  $c_{[\mathcal{O}_1],[\mathcal{O}_2]}:H^\bullet(X_{[\mathcal{O}_1]})\rightarrow H^\bullet(X_{[\mathcal{O}_2]})$.
\end{defn/thm}
\begin{proof}
Define $c_{[\mathcal{O}_1],[\mathcal{O}_2]}=c_{y_1,U,y_2}$, where $U$ is any nice little neighborhood of $y_1\in \mathcal{O}_1$ and $y_2\in U\cap\mathcal{O}_2$. By Theorem~\ref{Thm:SpecializationDef} this map is well-defined.
\end{proof}
\begin{prop}\label{Prop:TransitivitySpecialization}
If $\mathcal{O}_1,\mathcal{O}_2,\mathcal{O}_3\subset Y$ are $\GG$-orbits such that $\mathcal{O}_1\subset \overline{\mathcal{O}_2}\subset \overline{\mathcal{O}_3}$  then $c_{[\mathcal{O}_1],[\mathcal{O}_3]}=c_{[\mathcal{O}_2],[\mathcal{O}_3]}\circ c_{[\mathcal{O}_1],[\mathcal{O}_2]}$.
\end{prop}
\begin{proof}
Let $y_1\in\mathcal{O}_1$. Let $U$ be a nice little neighborhood of $y_1$. Let $y_2\in \mathcal{O}_2\cap U$. Let $V\subseteq U$ be a nice little neighborhood of $y_2$, which exists by Lemma~\ref{Lem:OpenNeigh}(i). Let $y_3\in \mathcal{O}_3\cap V$. We consider the following diagram 
$$
\xymatrix@C=10pt{
H^\bullet(\pi^{-1}(y_1))\ar@{..>}^{c_{y_1,U,y_2}}[rrrr]\ar@{..>}@/_30pt/_{c_{y_1,U,y_3}}[ddrr]&&&&H^\bullet(\pi^{-1}(y_2))\ar@{..>}@/^30pt/^{c_{y_2,V,y_3}}[ddll]\\
&H^\bullet(\pi^{-1}(U))\ar^\cong[ul]\ar[rr]\ar[rrru]\ar[dr]&&H^\bullet(\pi^{-1}(V))\ar_\cong[ur]\ar[dl]&\\
&&H^\bullet(\pi^{-1}(y_3))&&
}
$$
whose solid maps are induced by inclusions. It follows that the solid triangles commute. This proves $c_{y_1,U,y_3}=c_{y_2,V,y_3}\circ c_{y_1,U,y_2}$ and thus, by Definition~\ref{Def:Specialization}, the required equality.
\end{proof}
\begin{rmk}
Proposition~\ref{Prop:TransitivitySpecialization} can be reformulated as follows: let $\mathcal{P}$ be the category whose objects are the $\GG$-orbits in $Y$ and morphisms are inclusions of the orbit closures. Then specialization defines a functor $C:\mathcal{P}\rightarrow GrAlg$ (where $GrAlg$ denotes the category of graded $\ZZ$-algebras) defined on objects as $\mathcal{O}\mapsto H^\bullet(X_{[\mathcal{O}]})$ and on morphisms $\overline{\mathcal{O}_1}\subset \overline{\mathcal{O}_2}\mapsto c_{[\mathcal{O}_1],[\mathcal{O}_2]}$. In fact $\mathcal{P}$ is the category attached to a poset $(\mathcal{P},\leq)$ and by abuse of language we may consider $\mathcal{P}$ either as a category or as a poset.
\end{rmk}

\begin{prop}\label{equiv}
The following statements are equivalent:
\begin{enumerate}
    \item[(i)] for every cover 
    $\mathcal{O}_1\subset\overline{\mathcal{O}_2}$ in $\mathcal{P}$, the specialization map $c_{[\mathcal{O}_1],[\mathcal{O}_2]}$ is surjective;
    \item[(ii)] for every  
    $\mathcal{O}_1\subset\overline{\mathcal{O}_2}$, the specialization map $c_{[\mathcal{O}_1],[\mathcal{O}_2]}$ is surjective.
   \end{enumerate}
\end{prop}

\begin{proof}
It follows by Proposition~\ref{Prop:TransitivitySpecialization}.
\end{proof}
We discuss now the relative case in which we compare the specialization maps for the family $\pi:X\rightarrow Y$ with specialization maps for a subfamily $\pi':X'\rightarrow Y'$.

\begin{prop}\label{Prop:RestrictionSpecialization}
Let $Y'\subset Y$ be a closed subvariety acted upon by a subgroup $G'\subset G$. Let $X'=\pi^{-1}(Y')$ and let $\pi'=\pi|_{X'}:X'\rightarrow Y'$. Suppose that the datum $(G', X', Y',\pi')$ satisfies the geometric setting \eqref{Def:SpecializationG}-\eqref{Def:SpecializationProper}. Let $\mathcal{O}'_1,\mathcal{O}'_2\subset Y'$ be $\GG'$-orbits such that $\mathcal{O}'_1\subset \overline{\mathcal{O}'_2}$. Let $\mathcal{O}_1:= G\cdot \mathcal{O}'_1$ and $\mathcal{O}_2:=G\cdot \mathcal{O}'_2$ be the two $\GG$-orbits in $Y$ such that $\mathcal{O}'_1\subset \mathcal{O}_1$ and $\mathcal{O}'_2\subset \mathcal{O}_2$. Then there is an equality of specialization maps
$$
c_{[\mathcal{O}_1],[\mathcal{O}_2]}=c_{[\mathcal{O}'_1],[\mathcal{O}'_2]}
$$ 
where $c_{[\mathcal{O}_1],[\mathcal{O}_2]}$ is defined by $\pi$ and $c_{[\mathcal{O}'_1],[\mathcal{O}'_2]}$ is defined by $\pi'$.
\end{prop}
\begin{proof}
Let $y_1\in \mathcal{O}'_1$ and  $U\subset Y$ be a nice little neighborhood of $y_1$ in $Y$. Since $U\cap Y'$ is an open neighborhood of $y_1$ in $Y'$, by Lemma~\ref{Lem:OpenNeigh}, there exists a  nice little neighborhood $U'\subset U\cap Y'$ of $y_1$ in $Y'$.  Let $y_2\in U'\cap \mathcal{O}'_2\subset U\cap \mathcal{O}_2$. Consider the specialization map (for $\pi$) 
$$
c_{y_1,U,y_2}:H^\bullet(X_{[\mathcal{O}_1]})\cong H^\bullet(\pi^{-1}(y_1))\rightarrow H^\bullet(\pi^{-1}(y_2))\cong H^\bullet(X_{[\mathcal{O}_2]})
$$
and the specialization map (for $\pi'$)
$$
c_{y_1,U',y_2}:H^\bullet(\pi'^{-1}(y_1))\rightarrow H^\bullet(\pi'^{-1}(y_2)).
$$
The commutativity of the following diagram 
$$
\xymatrix{
H^\bullet(\pi^{-1}(y_1))\ar@{=}[d]&H^\bullet(\pi^{-1}(U))\ar_\cong[l]\ar[r]\ar[d]&H^\bullet(\pi^{-1}(y_2))\ar@{=}[d]\\
H^\bullet(\pi'^{-1}(y_1))&H^\bullet(\pi'^{-1}(U'))\ar_\cong[l]\ar[r]&H^\bullet(\pi'^{-1}(y_2))
}
$$
implies $c_{y_1,U,y_2}=c_{y_1,U',y_2}$, proving the statement.
\end{proof}
An $\alpha$-partition of a variety $X$ is a decomposition $X=\coprod U_j$ into locally closed subsets (for the Zariski topology) such that the pieces can be ordered so that $U_1\coprod\cdots\coprod U_i$ is closed for every $i$. An affine paving of $X$ is an $\alpha$-partition whose parts $U_j$ are isomorphic to affine spaces (see \cite{DCLP}). 
\begin{defn}\label{Def:CellDecFamily}
An \emph{$\alpha$-partition} of a 
proper $\CC^\times$-equivariant family $p:\XX\rightarrow \CC$ is an $\alpha$-partition $\XX=\coprod U_j$ of $\XX$ with the following property: each part $U_j$ is  $\CC^\times$-stable and if $p(U_j)\neq \{0\}$ then the restriction $p|_{U_j}:U_j\rightarrow\CC$ is a fiber bundle (hence trivial). An $\alpha$-partition of $p$ is called an \emph{affine paving} if each part is isomorphic to an affine space, i.e. $U_j$ is an affine space and, if $p(U_j)\neq \{0\}$, there is an isomorphism $U_j\simeq \CC^{\ell_j}\times  \CC$ such that the diagram
$$
\xymatrix{
U_j\ar^(.4){\simeq}[r]\ar_{p|_{U_j}}[d]&\CC^{\ell_j}\times \CC\ar^{p_2}[d]\\
\CC&\CC\ar@{=}[l]
}
$$
commutes.  
\end{defn}
\begin{rmk}\label{Rem:SpecForFamiliesCellDec}
A proper $\CC^\times$-equivariant  family $p:\XX\rightarrow \CC$ satisfies the geometric setting, i.e. the datum $(G=\CC^\times, X=\XX, Y=\CC, \pi=p)$ satisfies \eqref{Def:SpecializationG}-\eqref{Def:SpecializationProper}. Thus, by Definition/Theorem~\ref{Def:Specialization} there is a canonical specialization map $c_p:H^\bullet(p^{-1}(0))\rightarrow H^\bullet(p^{-1}(1))$. It is worth noticing that, by $\CC^\times$-equivariance, $U=\CC$ itself is a nice little neighborhood of $0$, that is,  the inclusion $p^{-1}(0)\subset\XX$ induces an isomorphism $H^\bullet(\XX)\simeq H^\bullet(p^{-1}(0))$. 
\end{rmk}

\begin{prop}\label{Prop:OneDimFamily}
Let $p:\XX\rightarrow \CC$ be a proper $\CC^\times$-equivariant family which admits an affine paving $\XX=\coprod_{j=1}^N U_j$. Then the specialization map $c_p:H^\bullet(p^{-1}(0))\rightarrow H^\bullet(p^{-1}(1))$ is surjective. 
\end{prop}
\begin{proof}
We order the cells $(U_j)$ so that $V_i=U_1\coprod\cdots \coprod U_i$ is closed for every $i$. We denote by $p_i=p|_{V_i}:V_i\rightarrow \CC$  the restricted family. Since $p_i$ is a proper $\CC^\times$-equivariant family which admits an affine paving, by Remark~\ref{Rem:SpecForFamiliesCellDec} there is a canonical specialization map $c_{p_i}:H^\bullet(p_i^{-1}(0))\rightarrow H^\bullet(p_i^{-1}(1))$ and $c_{p}=c_{p_N}$. Let us show by induction on $i$, that $c_{p_i}$ is surjective for all $i$.
We denote by $\XX_0=p^{-1}(0)$ and by $\XX_1=p^{-1}(1)$. We notice that $\XX_0\cap V_i=\coprod_{j=1}^i (U_j\cap \XX_0)$ and $\XX_1\cap V_i=\coprod_{j=1}^i (U_j\cap \XX_1)$ are $\alpha$-partitions. By hypothesis, $U_j\cap \XX_0$ is a cell and $U_j\cap \XX_1$ is either empty (when $U_j\subset \XX_0$) or it is a cell.  If $i=1$, since $\XX_0$  and $\XX_1$ are closed subsets of a projective space and $\XX_0\cap V_1$  and $\XX_1\cap V_1$ are closed cells, then $\XX_0\cap V_1$ is a point and $\XX_1\cap V_1$ is either empty or a point; in both cases the specialization map is surjective.  Let $i>1$ and denote by $X=V_i$, $U=U_i$ and $Z=V_i\setminus U_i=\coprod_{j=1}^{i-1}U_j$. For $k=0,1$, we denote by $X^k=X\cap\XX_k$, $U^k=U\cap \XX_k$ and $Z^k=Z\cap\XX_k$. We have the following commutative diagram:
$$
\xymatrix{
H^\bullet(X^0,Z^0)\ar[r]&H^\bullet(X^0)\ar[r]\ar@/^3pc/@{..>}^(.4){c_{p_{i}}}[dd]&H^\bullet(Z^0)\ar@/^3pc/@{..>}^(.4){c_{p_{i-1}}}[dd]\\
H^\bullet(X,Z)\ar[r]\ar^\simeq[u]\ar[d]&H^\bullet(X)\ar[r]\ar^\simeq[u]\ar[d]&H^\bullet(Z)\ar^\simeq[u]\ar^{f}[d]\\
H^\bullet(X^1,Z^1)\ar[r]&H^\bullet(X^1)\ar[r]&H^\bullet(Z^1)
}
$$
By Remark~\ref{Rem:SpecForFamiliesCellDec}, the upper vertical morphisms are isomorphisms. By induction $c_{p_{i-1}}$ and hence $f$  is surjective. 
Since $Z^0$ admits an affine paving, it has no odd cohomology. Moreover $U^0=X^0\setminus Z^0$ is an affine space, thus homeomorphic to the open unit ball $\BBB\setminus \Sp\subset \RR^{2\ell}$ where $\Sp$ denotes the unit sphere of  $\RR^{2\ell}$ and $\BBB$ denotes the closed unit ball of  $\RR^{2\ell}$.  

It is well-known that $H^\bullet(X^0,Z^0)=H^\bullet(\BBB,\Sp)$, and in particular  $H^\bullet(X^0,Z^0)$ is concentrated in a unique even degree.  Indeed,  by \cite[Prop. 2.22]{Hatcher}, the relative cohomology of the pair $H^\bullet(X^0,Z^0)$ is the reduced cohomology of the space (endowed with the quotient topology) obtained from $X^0$ by collapsing $Z^0$ to a point. Since $X^0$ is compact and $Z^0\subset X^0$ is closed, one verifies that this new space is homeomorphic to the one-point compactification of $X^0\setminus Z^0$, which is also obtained from $\BBB$ by collapsing $\Sp$ to a point. 

Since $U^1$ is either empty or isomorphic to $U^0$, the same argument shows that  $H^\bullet(X^1,Z^1)$  is concentrated in a unique even degree, if non-zero. 
It follows that the diagram above is a commutative diagram of short exact sequences: 
$$
\xymatrix{
0\ar[r]&H^\bullet(X^0,Z^0)\ar[r]&H^\bullet(X^0)\ar[r]&H^\bullet(Z^0)\ar[r]&0\\
0\ar[r]&H^\bullet(X,Z)\ar[r]\ar^\simeq[u]\ar^g[d]&H^\bullet(X)\ar[r]\ar^\simeq[u]\ar[d]&H^\bullet(Z)\ar^\simeq[u]\ar@{->>}[d]\ar[r]&0\\
0\ar[r]&H^\bullet(X^1,Z^1)\ar[r]&H^\bullet(X^1)\ar[r]&H^\bullet(Z^1)\ar[r]&0
}
$$
We prove that $g$ is surjective. If $p(U)=\{0\}$ then  $X^1=Z^1$ and hence  $H^\bullet(X^1,Z^1)=0$; in this case $g=0$ is surjective. If $p(U)=\CC$ then $U\simeq (\BBB\setminus\Sp)\times\CC$. Again using \cite[Prop. 2.22]{Hatcher}, one gets 
$$
H^\bullet(X,Z)\simeq H^\bullet(\BBB\times\CC,\Sp\times \CC)\simeq H^\bullet(\BBB,\Sp)\otimes H^\bullet(\CC)\simeq H^\bullet(\BBB,\Sp).
$$
Indeed, one checks that the space obtained from $X$ by collapsing $Z$ to a point is homeomorphic to the space obtained from $\BBB\times \CC$ by collapsing $\Sp\times \CC$ to a point.

It follows that we get a commutative diagram whose horizontal maps are isomorphisms:
$$
\xymatrix{
H^\bullet(X,Z)\ar^(.4){\simeq}[r]\ar^g[d]&H^\bullet(\BBB\times \CC, \Sp\times \CC)\ar^{g'}[d]\\
H^\bullet(X^1,Z^1)\ar^(.4){\simeq}[r]&H^\bullet(\BBB\times \{1\}, \Sp\times \{1\}).
}
$$ 
By \cite[Proposition~VII.7.6]{D},  $g'$ is an isomorphism. 
\end{proof}

\section{Specialization map for quiver Grassmannians}\label{Sec:SpecQuivGrass}

We now apply the results of Section~\ref{Sec:Specialization} to study specialization map for quiver Grassmannians of Dynkin type. We hence fix a Dynkin quiver $\Q$ and two dimension vectors $\mathbf{e,d}\in\ZZ^{\Q_0}_{\geq0}$. We consider the universal quiver Grassmannian $\pi_{\ee,\dd}:\XXX_{\ee,\dd}\rightarrow \mathrm{R}_\dd$ defined in Section~\ref{SubSec:QuivGrass}. 
\begin{lem}\label{Lem:GeomSettingQuivGrass}
The datum $(\GG_\dd, \XX_{\ee,\dd}, \mathrm{R}_\dd,\pi_{\ee,\dd})$ satisfies the geometric setting  \eqref{Def:SpecializationG}-\eqref{Def:SpecializationProper}.
\end{lem}
\begin{proof}
Conditions \eqref{Def:SpecializationG} and \eqref{Def:SpecializationGVar} are obvious. Condition \eqref{Def:SpecializationFiniteGOrbits} follows by Gabriel's theorem. Unibranchness of orbit closures is proved by Zwara \cite[Theorem~1.2]{Zwara-unibranch}; thus \eqref{Def:SpecializationUnibranch} is satisfied. The stabilizer of a $\Q$-representation $M$ is $\mathrm{Aut}_\Q(M)$ which is a Zariski-open subset of the vector space $\mathrm{End}_\Q(M)$; thus \eqref{Def:SpecializationStab} holds. The last condition \eqref{Def:SpecializationProper} holds because $\pi_{\ee,\dd}$ is a closed sub-family of the constant family $\Gr_\ee(\CC^\dd)\times \mathrm{R}_\dd\rightarrow \mathrm{R}_\dd$, which is proper.
\end{proof}

Recall that  $\GG_\dd$-orbits  in $\mathrm{R}_\dd$ correspond to isomorphism classes of $\Q$-representations. Given an isomorphism class $[M]$ of $\Q$-representations corresponding to an orbit $\mathcal{O}\subset \mathrm{R}_\dd$, we denote by $H^\bullet(\Gr_\ee([M]))$ as in Definition~\ref{Def:HSpecialization}, i.e., as $H^\bullet(X_{[\mathcal{O}]})$ for $X=\XXX_{\ee,\dd}$.
By Lemma~\ref{Lem:GeomSettingQuivGrass} and Theorem/Definition~\ref{Def:Specialization}, given two isomorphism classes $[N]$ and $[M]$ corresponding to the two orbits $\mathcal{O}_1$ and $\mathcal{O}_2$ in $\mathrm{R}_\dd$, respectively, assuming that $\mathcal{O}_1\subset\overline{\mathcal{O}_2}$ there is a well-defined specialization map 
$$
c_{[N],[M]}: H^\bullet(\Gr_\ee([N]))\rightarrow H^\bullet(\Gr_\ee([M]))
$$
where, $c_{[N],[M]}:=c_{[\mathcal{O}_1],[\mathcal{O}_2]}$ (see Definition~\ref{Def:Specialization}). 

Let $N_0$ be the unique semi-simple isomorphism class of dimension vector $\dd$. Let $0_\dd\in \mathrm{R}_\dd$ be the point corresponding to $N_0$. It is the zero vector of the complex vector space $\mathrm{R}_\dd$ and it is the unique $\GG_\dd$-fixed point of $\mathrm{R}_\dd$. 
For any isomorphism class $[M]$ of dimension vector $\dd$, let $\mathcal{O}$ denote the corresponding $\GG_\dd$-orbit in $\mathrm{R}_\dd$. Let $m\in \mathcal{O}$ . Then there is a closed embedding $\iota: \xymatrix{\Gr_\ee(m)=\pi_{\ee,\dd}^{-1}(m)\ar@{^(->}[r]&\pi_{\ee,\dd}^{-1}(0_\dd)}=\prod_{i\in \Q_0}\Gr_{\ee_i}(\CC^{\dd_i})$ which gives a morphism in cohomology 
$$
\xymatrix{
r_{m}:H^\bullet(\Gr_\ee(0_\dd))\ar^(.55){\iota^\ast}[r]&H^\bullet(\Gr_\ee(m))\ar^\cong[r]& H^\bullet(\Gr_\ee([M]))
}
$$

\begin{prop}\label{semisimple}
With the notation introduced above the following statements hold:
\begin{enumerate}
\item[(i)] $r_{m_1}=r_{m_2}$, for any  $m_1, m_2\in \mathcal{O}_M$. 
\item[(ii)] $r_m=c_{0_\dd,M}$, for any $m\in \mathcal{O}$.
\end{enumerate}
Thus, we may define the map $r_M:=r_m$ for any $m\in \mathcal{O}_M$. 
\end{prop}

\begin{proof}
Let us prove $(i)$. Let $g\in \GG_\dd$ be such that $g.m_2=m_1$. The commutativity of the following diagram 
$$
\xymatrix{
&H^\bullet(\Gr_\ee(0_\dd))\ar^{\iota^\ast}[r]\ar^{\varphi_{0_\dd,0_\dd,g}=Id}[dd]&H^\bullet(\Gr_\ee(m_1))\ar^{\varphi_{m_1,m_2,g}}[dd]\ar^\cong[dr]& \\
H^\bullet(\Gr_\ee(N_0))\ar[ur]\ar[dr]&&&H^\bullet(\Gr_\ee([M]))\\
&H^\bullet(\Gr_\ee(0_\dd))\ar^{\iota^\ast}[r]&H^\bullet(\Gr_\ee(m_2))\ar^\cong[ur]&
}
$$ 
yields the equality of the upper and lower paths, that is $r_{m_1}=r_{m_2}$.

To prove $(ii)$, we consider a nice little neighborhood $U$ of  $0_\dd$. Let $m\in U\cap \mathcal{O}_M$. Let $i:\{m\}\rightarrow U$ be the natural inclusion. Since $U$ is contractible, the lower right map of the following diagram is an isomorphism. 
$$
\xymatrix{
H^\bullet(\pi_{\ee,\dd}^{-1}(m))&H^\bullet(\pi_{\ee,\dd}^{-1}(U))\ar[l]\ar^\cong[r]&H^\bullet(\pi_{\ee,\dd}^{-1}(0_\dd))\ar@/_20pt/_{c_{0_\dd,U,m}}[ll]\\
H^\bullet(\mathrm{Gr}_{\mathbf{e}}(\CC^\dd)\times\{m\})\ar^{i^\ast}[u]&\ar_\cong[l]H^\bullet(\mathrm{Gr}_{\mathbf{e}}(\CC^\dd)\times U)\ar^\cong[r]\ar[u]&H^\bullet(\mathrm{Gr}_{\mathbf{e}}(\CC^\dd)\times \{0_\dd\})\ar@{=}[u]\ar@/^20pt/^{Id}[ll]
}
$$ 
Since this diagram commutes, the statement follows.
\end{proof}

\begin{prop}\label{equiv2}
Let $\mathcal{Q}$ be a quiver of Dynkin type. The following statements are equivalent:
\begin{enumerate}
    \item[(i)] for every minimal degeneration 
    $M\leq_{\mathrm{deg}} N$, the specialization map $c_{N, M}$ is surjective;
    \item[(ii)] for every degeneration 
    $M\leq_{\mathrm{deg}} N$, the specialization map $c_{N, M}$ is surjective;
    \item[(iii)] for every $M\in \mathrm{R}_{\mathbf{d}}$, the map $r_M:\bigotimes_{i\in\Q_0} H^\bullet(\Gr_{\mathbf{e}_i}(\CC^{\mathbf{d}_i})\rightarrow H^\bullet(\Gr_\mathbf{e}(M))$ is surjective. 
\end{enumerate}
\end{prop}

\begin{proof}
The equivalence of statements (i) and (ii) follows by Proposition~\ref{equiv}.

By Proposition~\ref{semisimple}, statement (iii) is equivalent to the surjectivity of the specialization map $c_{0_\dd ,M}$. It follows that (ii) implies (iii). For any $M\leq_{\mathrm{deg}} N$, by Proposition~\ref{Prop:TransitivitySpecialization}, 
$$
r_M=c_{0_\dd , [M]} = c_{[N],[M]} \circ c_{0_\dd, [N]}.
$$
Thus, the surjectivity of $r_M$ implies the surjectivity of $c_{N , M}$. Hence (iii) implies (ii).
\end{proof}

\subsection{One-dimensional families associated with non-split short exact sequences}\label{Subsec:FamilySES}

Let $\xi:\xymatrix@C=10pt{0\ar[r]& X \ar[r]& Y \ar^p[r]& S\ar[r]&0}$ be a non-split short exact sequence in $\mathbf{Rep}(\mathcal{Q})$. Fixing bases of $X$ and $S$ and a section of the map $p:Y \rightarrow S$ as graded vector spaces, one gets two distinct points $M, N\in \mathrm{R}_{\mathbf{d}}$ corresponding respectively to $M=Y$ and $N=X\oplus S$, where $\mathbf{d} = \dim(Y)$. With respect to this basis, the linear maps of $N=X\oplus S$ are block-diagonal 
    $$
    \begin{pmatrix}
    f_\alpha & 0 \\
    0 & g_\alpha
    \end{pmatrix}
    $$
    and the linear maps of $M=Y$ are block-upper triangular, with block-diagonal part coinciding with that of $N$
    $$
    \begin{pmatrix}
    f_\alpha & h_\alpha \\
    0 & g_\alpha
    \end{pmatrix}.
    $$
    We consider the affine map $\mathcal{L}_{N,M}:\CC\rightarrow\mathrm{R}_{\mathbf{d}}$ given by
    \begin{equation}\label{Eq:MatrixXLambda}
    \lambda \mapsto
    \begin{pmatrix}
    f_\alpha & \lambda h_\alpha \\
    0 & g_\alpha
    \end{pmatrix}, \ \ \lambda\in\mathbb{C}.
    \end{equation}
    We consider the one-dimensional sub-torus  $\CC^\times$ of $G_\mathbf{d}$ which in this basis is given by 
    \begin{equation}\label{Eq:OneDimTorus}
    t_\lambda=
    \begin{pmatrix}
    \lambda & 0 \\
    0 & 1
    \end{pmatrix}, \ \ \lambda\in\mathbb{C}^\ast.
    \end{equation}
We notice that $t_\lambda\cdot M=(t_{\lambda}M_\alpha t_{\lambda^{-1}})_{\alpha\in\Q_1}=\mathcal{L}_{N,M}(\lambda)$ and $t_\lambda\cdot N=N$, i.e. $\mathcal{L}_{N,M}$ is $\CC^\times$-equivariant, where the action of $\CC^\times$ on $\CC$ is by multiplication.

\begin{defn}\label{defonedim}
We define the family $p_\xi$ associated to $\xi$ as the pull-back: 
$$
\xymatrix{
\XXX_\xi\ar[r]\ar_{p_\xi}[d]&\XXX_{\mathbf{e,d}}\ar^{\pi_{\mathbf{e,d}}}[d]\\
\CC\ar^{\mathcal{L}_{N,M}}[r]&\mathrm{R}_{\mathbf{d}}
}
$$
\end{defn} 

The following statements hold.

\begin{prop}
The proper family $p_\xi:\XXX_\xi\rightarrow\CC$ is acted upon by the torus $\CC^\times$ defined in \eqref{Eq:OneDimTorus}. Thus $p_\xi:\XXX_\xi\rightarrow\CC$ is a proper $\CC^\times$-equivariant family. 
\end{prop}

\begin{proof}
Since   $\mathcal{L}_{N,M}$ and $\pi_{\mathbf{e,d}}$ are  $\CC^\times$-equivariants,  the pull-back  $p_\xi$ is $\CC^\times$-equivariant. 
\end{proof}

\begin{rmk}\label{Rem:XILambda}
For $\lambda\in\CC$ we denote by $\xi_\lambda$ the short exact sequence $\xymatrix@C=10pt{0\ar[r]& X \ar^\iota[r]& Y_\lambda\ar^p[r]& S\ar[r]&0}$ where $Y_\lambda$ is the quiver representation defined by \eqref{Eq:MatrixXLambda}. We notice that the inclusion $\iota$ and the projection $p$ as maps of the graded vector space $X\oplus S$ do not  depend on $\lambda$.
\end{rmk}

\subsection{\texorpdfstring{$\alpha$}{alpha}-partitions for one-dimensional families associated to generating extensions}\label{subsec:ALfaPartOneDimFam}
   In this section we consider non-split \emph{generating} short exact sequences (see Section~\ref{Sec:generating}).   We prove that the family $p_\xi:\mathcal{X}_\xi\rightarrow\CC$ associated to a generating non-split extension admits an $\alpha$-partition (see Definition~\ref{Def:CellDecFamily}) which is an affine paving if $\Q$ is of type $A$.
 
 We thus fix  a non-split generating  short exact sequence $\xi:\xymatrix@C=10pt{0\ar[r]& X \ar[r]& Y \ar[r]& S\ar[r]&0}.$   
    We think of  $X\oplus S$ and $Y$ as having the same underlying graded vector space $V$ and thus  $X$ is a graded vector subspace of $V$. In particular for a point $(A,Y_\lambda)$ of $\mathcal{X}_\xi$ it makes perfect sense to consider the quiver representations $A\cap X$, $A/A\cap X$ and $X/A\cap X$.
    
    \begin{defn}\label{alfa}
Let $\mathcal{S}_{\mathbf{f}, \mathbf{g}}^{(i)}$, with $\mathbf{f}+\mathbf{g}=\mathbf{e}$ and $i=0,1$, be the subset of $\mathcal{X}_\mathcal{L}$ defined as follows:
       
       $$
       \mathcal{S}_{\mathbf{f}, \mathbf{g}}^{(i)} = 
       \left\{ 
       (A, Y_\lambda) \in \XXX_\LL\ \ | \ \ \mathbf{dim}( A\cap X) = \mathbf{f} \ \mathrm{and} \ [A/A\cap X , X/A\cap X]^1 = i
       \right\}
       $$
    \end{defn}
    
    \begin{lem}\label{Lem:AlphaPart}
    The partition $\XXX_\xi= \coprod_{\ff, \Gg, i} \SSS^{(i)}_{\ff, \Gg}$ is an $\alpha$-partition.
    \end{lem}
    \begin{proof}
    Let us introduce the coarser partition $\XXX_\xi= \coprod_{\ff, \Gg} \SSS_{\ff, \Gg}$ where
     $\SSS_{\ff, \Gg}:=\SSS_{\ff, \Gg}^{(0)}\cup \SSS_{\ff, \Gg}^{(1)}$. This is a finite partition and for every dimension vector $\mathbf{n}$ the set $\coprod_{\ff+\Gg\geq \mathbf{n}} \SSS_{\ff, \Gg}$ is closed (here $\geq$ is meant componentwise). It is thus an $\alpha$-partition. By upper semi-continuity of $\dim (\Ext^1 (-, -))$, the partition $\SSS_{\ff, \Gg}:=\SSS_{\ff, \Gg}^{(0)}\cup \SSS_{\ff, \Gg}^{(1)}$ is an $\alpha$-partition with $\SSS_{\ff, \Gg}^{(1)}$ closed. Combining these two statements together one gets the result.   
        \end{proof}
\begin{rmk}
The proof of Lemma~\ref{Lem:AlphaPart} uses the same argument used in \cite[Lemma~20]{CEFR} for a single fiber. 
\end{rmk}   
We consider the following commutative diagram of maps
$$
\begin{tikzcd}
\SSS^{(i)}_{\ff , \Gg} \arrow{r}{{j_{\ff,\Gg}^{(i)}}} \arrow{d}{\Psi_{\ff, \Gg}^{(i)}}
& \XXX_\xi\arrow{d}{\pi_\xi}\\
\Gr_\ff (X) \times \Gr_\Gg (S) \times \CC \arrow{r}{p_3} & \CC
\end{tikzcd}
$$
where {$j_{\ff,\Gg}^{(i)}$} is the natural inclusion, the map $p_3$ is the projection on the third factor, and $\Psi_{\ff, \Gg}^{(i)}$ maps $(A,Y_\lambda)$ to 
$(A\cap X , p(A), Y_\lambda)$. We notice that $\SSS^{(i)}_{\ff , \Gg}$ is $\CC^\times$-stable; furthermore, {noticing} that $\CC^\times$ acts on $\Gr_\ff (X) \times \Gr_\Gg (S) \times \CC$ by its natural action on $\CC$, one has that $\Psi_{\ff, \Gg}^{(i)}$ and $p_3$ are $\CC^\times$-equivariant. 

Now we study the parts $\SSS^{(i)}_{\ff , \Gg}$ of the $\alpha$-partition. Our analysis is based on the techniques of \cite[Section 3]{CEFR}.

\begin{lem}\label{l1}
The map $\Psi_{\ff, \Gg}^{(i)}$ is a locally trivial affine bundle on its image. Moreover, the image of $\Psi_{\ff, \Gg}^{(i)}$ is
\begin{enumerate}
\item[(i)] $((\Gr_\ff (X)\times \Gr_\Gg (S) )-
(\Gr_\ff (X_S)\times \Gr_{\Gg -\dim(S^X)}(S/S^X)))
\times \CC$, if $i=0$; and
\item[(ii)] $
\Gr_\ff (X_S)\times \Gr_{\Gg -\dim(S^X)}(S/S^X)
\times \{0\}$, if $i=1$.
\end{enumerate}
\end{lem}

\begin{proof}
By \cite[Theorem~26]{CEFR}, if there exists $N\subset Y_\lambda$ such that $[p(N),X/N\cap X]^1=1$ then $\lambda=0$. It follows that $\Psi_{\ff, \Gg}^{(1)}$ has image in $\Gr_\ff (X)\times \Gr_\Gg (S)\times \{Y_0\}$. The stated image follows from \cite[Lemma~31\eqref{Def:SpecializationUnibranch}]{CEFR}.
Following \cite[Section 3]{CEFR}, one has two vector bundles $H$ and $K$ on 
$$
\Gr_\ff (X) \times \Gr_\Gg (S)\times \CC
$$
whose fiber over a point 
$(N_1,N_2,\lambda)$ is 
$$\prod_{i\in \Q_0} \Hom((N_2)_i, (X/N_1)_i)\times \{\lambda\}$$ 
and 
$$\prod_{\alpha\in \Q_1} \Hom((N_2)_{s(\alpha)}, (X/N_1)_{t(\alpha)})\times \{\lambda\}$$
respectively. Let $\Phi: H \rightarrow K$ be the Ringel map of vector bundles (see \cite[page~14]{CEFR}) whose kernel over $(N_1,N_2,\lambda)$ is $\Hom_\Q(N_2,X/N_1)\times\{\lambda\}$. By Remark~\ref{Rem:XILambda}, every splitting $\theta:S\rightarrow Y$ of the short exact sequence $\xi_1$ (as $\Q_0$-graded vector spaces)
 is also a splitting of $\xi_\lambda$, for all $\lambda\in\CC$. We choose a splitting $\theta$ so that $Y_\lambda$ is represented in suitable basis of the $\Q_0$-graded vector space $X\oplus S$ as 
$$
(Y_\lambda)_\alpha=\begin{pmatrix}f_\alpha&\lambda n_\alpha\\0&g_\alpha\end{pmatrix}
$$ 

 By \cite[Section 3]{CEFR}, one gets a section $z_\theta$ of $K$ over $\Gr_\ff (X) \times \Gr_\Gg (S)$ and thus . As in \cite[Lemma 22]{CEFR}, one shows that there is an isomorphism of schemes
$$
\SSS_{\ff , \Gg}
=
\SSS^{(0)}_{\ff , \Gg}
\coprod \SSS^{(1)}_{\ff , \Gg}
\cong \Phi^{-1}(z_\theta).
$$
Analogously to \cite[Theorem 24]{CEFR}, we find that the parts $\SSS^{(0)}_{\ff , \Gg}$ and $\SSS^{(1)}_{\ff , \Gg}$ are affine bundles on their {images under} $\Phi$. Moreover, as in \cite[Theorem 32]{CEFR}, one gets the {statement (i) of the lemma}.

{Statement (ii)} follows, using \cite[Lemma 31 \eqref{Def:SpecializationUnibranch}]{CEFR},  by observing that the image of $\SSS^{(1)}_{\ff , \Gg}$ through $\Phi$ is exactly 
$
    \Gr_\ff (X_S) 
    \times \Gr_{\Gg-\textbf{dim}\, S^X} (S/S^X)
    \times \{ N \}.
    $
\end{proof}
\begin{prop}
The $\alpha$-partition $\XXX_\xi= \coprod_{\ff, \Gg, i} \SSS^{(i)}_{\ff, \Gg}$ induces an $\alpha$-partition of the family $p_\xi:\XXX_\xi\rightarrow\CC$.
\end{prop}
\begin{proof}
The restriction $p|_{ \SSS^{(i)}_{\ff, \Gg}}:\SSS^{(i)}_{\ff, \Gg}\rightarrow\CC$ is $\CC^\times$-equivariant. We have $p(\SSS^{(0)}_{\ff, \Gg})=\CC$ and $p(\SSS^{(1)}_{\ff, \Gg})=\{0\}$. The restriction map $p|_{\SSS^{(0)}_{\ff, \Gg}}: \SSS^{(0)}_{\ff, \Gg}\rightarrow \CC$ is defined as the composition
$$
\xymatrix{
\SSS^{(0)}_{\ff, \Gg}\ar@{->>}_{\Psi_{\ff, \Gg}^{(0)}}[d]\ar^{p|_{\SSS^{(0)}_{\ff, \Gg}}}[dr]&\\
U\times \CC\ar_(.65){p_2}[r]&\CC
}
$$
where $U$ is described in Lemma~\ref{l1}. Since $\Psi_{\ff, \Gg}^{(0)}$ is locally trivial on its image by Lemma~\ref{l1} and $p_2$ is trivial, we conclude that $p|_{\SSS^{(0)}_{\ff, \Gg}}$ is locally trivial and hence trivial. 
\end{proof}
\begin{lem}\label{l2}
Let $\Q$ be a quiver of type $A$ and let $M\leq_{\mathrm{deg}} N$ be a minimal degeneration associated with the generating short exact sequence 
$
\xymatrix{
\xi:0\ar[r]&X\ar[r]&Y\ar[r]&S\ar[r]&0
}
$
as in Corollary~\ref{Cor:MinDegGenerating}.
Then {both}
$\Gr_\ff (X)\setminus\Gr_\ff (X_S)$ and 
$\Gr_\Gg (S)\setminus\Gr_{\Gg - \dim(S^X)}(S/S^X)$ admit {affine pavings}.
\end{lem}

\begin{proof}
By Corollary~\ref{Cor:MinDegGenerating} there are decompositions $X=X_1\oplus X'$ and $S=S_1\oplus S'$ and a generating short exact sequence
$$
\xymatrix@R=0pt{
0\ar[r]&X_1\ar[r]&Y_1\ar[r]&S_1\ar[r]&0\\
          &\oplus&\oplus&\oplus&\\
          &X'\ar@{=}[r]&X'&&\\
          &  &\oplus&\\
          &&S'\ar@{=}[r]&S'}
$$
Moreover, $X_S=X'\oplus (X_1)_{S_1}$ and $S^X=(S_1)^{X_1}$. We now prove that $\Gr_\ff (X)\setminus\Gr_\ff (X_S)$ admits an affine paving. By Theorem~\ref{Thm:BongartzTypeA}(iv),  the split short exact sequence 
$$
\xymatrix@R=0pt{
\eta:&0\ar[r]&X'\ar^{\iota_\eta}[r]&X\ar^{p_\eta}[r]&X_1\ar[r]&0}
$$
is generating split. Thus, we get an $\alpha$-partition  
$\Gr_\ff(X)=\coprod_{\ff'+\ff''=\ff}\mathcal{S}_{\ff',\ff''}^\eta$  and vector bundles 
$$\xymatrix{\mathcal{S}_{\ff',\ff''}^\eta\ar@{->>}[r]&\Gr_{\ff'}(X')\times\Gr_{\ff''}(X_1)}, \ \ N\mapsto (\iota_\eta^{-1}(N),p_\eta(N)).$$ 
Since $X_1$ is indecomposable and $\Q$ is of type $A$, $\Gr_{\ff''}(X_1)$ is {either} empty or a point. We notice that $X_S=p_\eta^{-1}((X_1)_{S_1})$. We conclude that $\Gr_\ff (X)\setminus\Gr_\ff (X_S)$ is the union of those strata $\mathcal{S}_{\ff',\ff''}^\eta$ such that $\ff''\leq \mathbf{dim} (X_1)_{S_1}$. Since $\Gr_{\ff'}(X')\times\Gr_{\ff''}(X_1)$ admits an affine paving by \cite{CEFR}, the same holds for all strata $\mathcal{S}_{\ff',\ff''}^\eta$ and thus for $\Gr_\ff (X)\setminus\Gr_\ff (X_S)$. 

The proof of the fact that $\Gr_\Gg (S)\setminus\Gr_{\Gg - \dim(S^X)}(S/S^X)$ admits an affine paving is obtained using the generating split short exact sequence
$
0\rightarrow S_1 \rightarrow S \rightarrow S' \rightarrow 0
$ and proceeding analogously.
\end{proof}
\begin{prop}\label{Lem:5.6.1/2}
Let $\Q$ be a Dynkin quiver of type  $A$ and let $\xi:0\rightarrow X\rightarrow Y\rightarrow S\rightarrow 0$ be a generating extension which gives rise to a minimal degeneration $Y\leq_{\textrm{deg}}X\oplus S$. Then the family $p_\xi$ associated to $\xi$ admits an affine paving.  
\end{prop}
\begin{proof}
By Lemma~\ref{l1}, one has that $$\SSS^{(0)}_{\ff, \Gg}\rightarrow ((\Gr_\ff (X)\times \Gr_\Gg (S) )-
(\Gr_\ff (X_S)\times \Gr_{\Gg -\dim(S^X)}(S/S^X)))\times \CC$$ and  $$\SSS^{(1)}_{\ff, \Gg}\rightarrow \Gr_\ff (X_S) \times \Gr_{\Gg-\mathbf{dim}\,S^X} (S/S^X)$$ are locally trivial affine bundles. The quiver Grassmannians $\Gr_\ff (X)$, $\Gr_\Gg (S)$, $\Gr_\ff (X_S)$ and $\Gr_{\Gg-\mathbf{dim}\,S^X} (S/S^X)$ admit affine pavings by \cite[Theorem~46]{CEFR}; moreover the open subsets $\Gr_\ff (X)\setminus\Gr_\ff (X_S)$ and 
$\Gr_\Gg (S)\setminus\Gr_{\Gg - \dim(S^X)}(S/S^X)$ admit affine pavings  by Lemma~\ref{l2}. It follows  that the bases of those affine bundles admit affine pavings. It follows that for $i=0,1$, the family $p_\xi|_{\SSS^{(i)}_{\ff, \Gg}}:\SSS^{(i)}_{\ff, \Gg}\rightarrow\CC$ admits an affine paving. Since an $\alpha$-partition of each part of an $\alpha$-partition gives rise to an $\alpha$-partition, we conclude that the whole family $p_\xi$ admits an affine paving.
\end{proof}

\begin{rmk}
The proof of Proposition~\ref{Lem:5.6.1/2} is based on Lemma~\ref{l2}. In particular, Proposition~\ref{Lem:5.6.1/2} holds for generating short exact sequences of a Dynkin quiver (not necessarily of type $A$) for which the analogous statement of Lemma~\ref{l2} holds.  
\end{rmk}

\subsection{Proof of the main result }\label{Subsec:ProofMainResult}

In this section, we collect all arguments developed above to prove Theorem~\ref{Thm:MainIntro} of the introduction which is the main result of the paper. For convenience of the reader we restate it here in a slightly different but equivalent form. 
\begin{thm}\label{Thm:TypeA}
Let $\Q$ be a Dynkin quiver and let $[M]$ and $[N]$ be two isomorphism classes of $\Q$-representations of the same dimension vector $\dd$ such that $M\leq_{\textrm{deg}}N$.
\begin{enumerate}
\item[(i)] There are well-defined cohomology algebras $H^\bullet(\Gr_\ee([M]))$ and $H^\bullet(\Gr_\ee([N]))$.
\item[(ii)] There is a well-defined map of algebras $c_{[N],[M]}:H^\bullet(\Gr_\ee([N]))\rightarrow H^\bullet(\Gr_\ee([M]))$.
\item[(iii)] If $\Q$ is of type $A$ then $c_{[N],[M]}$ is surjective.
\end{enumerate}
\end{thm}
\begin{proof}
By Lemma~\ref{Lem:GeomSettingQuivGrass}, a universal quiver Grassmannian of Dynkin type satisfies the conditions  \eqref{Def:SpecializationG}-\eqref{Def:SpecializationProper} of the geometric setting. Then parts (i) and (ii) follow from 
Theorem/Definition~\ref{Def:HSpecialization} and  Theorem/Definition~\ref{Def:Specialization} respectively. We now prove part (iii). By Proposition~\ref{equiv} we can assume that $M\leq_{\textrm{deg}}N$ is a minimal degeneration. By Corollary~\ref{Cor:MinDegGenerating} the minimal degeneration is given by a generating short exact sequence $\xi:0\rightarrow X\rightarrow Y\rightarrow S\rightarrow 0$. By Proposition~\ref{Lem:5.6.1/2} the specialization map $c_{p_\xi}$ is surjective. By Proposition~\ref{Prop:RestrictionSpecialization} $c_{[N],[M]}=c_{p_{\xi}}$.
\end{proof}


\begin{thebibliography}{999}
\bibitem{AbeasisDelFraEquioriented}
S.~Abeasis, A.~Del~Fra. \emph{Degenerations for the representations of an equioriented quiver of type $A_m$}.Bull.~Un.~Mat.~Ital.~Suppl. (1980), n.~2, 157-171.
\bibitem{AbeasisDelFra}
S.~Abeasis, A.~Del~Fra. \emph{Degenerations for the representations of a quiver of type $A_m$}. J.~Algebra \textbf{93} (1982), 376-412.
\bibitem{ARS}
M.~Auslander, I.~Reiten, S.~Smalo. \emph{
Representation theory of Artin algebras}. Cambridge studies in advanced mathematics, \textbf{36} (1995).
\bibitem{ASS}
I.~Assem, D.~Simson, A.~Skowronski. \emph{
Elements of the representation theory of associative algebras. Vol. 1.
Techniques of representation theory}. London Mathematical Society Student Texts, \textbf{65} (2006). 
\bibitem{Bo}  Bongartz, Klaus. \emph{Minimal singularities for representations of Dynkin quivers}. Comment. Math. Helv. \textbf{69} (1994), no. 4, 575-611.
\bibitem{Bongartz}
K.~Bongartz. \emph{On Degenerations and Extensions of
Finite Dimensional Modules}. Adv. Math. \textbf{121} (1996), 245--287.
\bibitem{CC} P.~Caldero, F.~Chapoton. \emph{Cluster algebras as Hall algebras of quiver representations}. Comm. Math. Helv. \textbf{81} (2006), no. 3, 595-616.
\bibitem{CFR} G.~Cerulli~Irelli, E.~Feigin, M.~Reineke. \emph{Quiver Grassmannians and degenerate flag varieties}. Algebra  Number Theory \textbf{6} (2012).
\bibitem{CFFFR} G.~Cerulli~Irelli, X.~Fang, E.~Feigin, G.~Fourier, M.~Reineke. \emph{Linear degenerations of flag varieties}. Math.~Z. \textbf{287} (2017).
\bibitem{CFFFR2} G.~Cerulli~Irelli, X.~Fang, E.~Feigin, G.~Fourier, M.~Reineke. \emph{Linear degenerations of flag varieties: partial flags, defining equations, and group actions}. Math.~Z. \textbf{296} (2020).
\bibitem{CB}
W.~Crawley-Boevey, \emph{Lectures on representations of quivers}. Available on the author's webpage.
\bibitem{CEFR} G.~Cerulli~Irelli, Francesco Esposito, Hans Franzen, Markus Reineke.
\emph{Cell decompositions and algebraicity of cohomology for quiver Grassmannians}. Adv. Math. \textbf{379} (2021).
\bibitem{Ce} G.~Cerulli~Irelli. \emph{Three lectures on quiver Grassmannians}. Representation Theory and Beyond. Cont. Math. \textbf{758} (2020), 57-89.
\bibitem{CEsp} G.~Cerulli~Irelli, F.~Esposito. \emph{Geometry of quiver Grassmannians of Kronecker type and applications to cluster algebras}. Algebra Number Theory \textbf{5} (2011), no. 6, 777-801.
\bibitem{DCLP} C.~De Concini,  G.~Lusztig, C.~Procesi. \emph{Homology of the zero-set of a nilpotent vector field on a flag manifold}. J. Amer. Math. Soc. \textbf{1} (1988), no. 1, 15-34.
\bibitem{D} A.~Dold. \emph{Lectures on algebraic topology}, Classics in Mathematics, Springer-Verlag Berlin Heidelberg 1995.
\bibitem{FR} X.~Fang, M.~Reineke. \emph{Supports for linear degenerations of flag varieties}. Doc. Math. \textbf{26}(2021), 1981-2003.
\bibitem{GMP2} M.~Goresky, R.~MacPherson. \emph{Intersection homology II}. Inv. Math. \textbf{71}(1983), 77-129. 
\bibitem{Hatcher} A.~Hatcher. \emph{Algebraic Topology}. Cambridge University Press, Cambridge. 2002.
\bibitem{Kashiwara-Schapira}  M.~ Kashiwara, P.~Schapira. \emph{Sheaves on manifolds}. Springer-Verlag, Berlin, \textbf{292} (1994).
\bibitem{LaSt}  M.~Lanini, E.~Strickland. \emph{Cohomology of the flag variety under PBW degenerations}. Transform. Groups 24 (2019), no. 3, 835--844.
\bibitem{LaPu1} M.~Lanini, A. P\"utz. \emph{GKM-Theory for torus actions on cyclic quiver Grassmannians}. Algebra Number Theory \textbf{17} (2023), no. 12, 2055-2096.
\bibitem{LaPu2} M.~Lanini, A. P\"utz. \emph{Permutation actions on quiver Grassmannians for the equioriented cycle via GKM-theory}. J. Alg. Comb. \textbf{57} (2023), no. 3, 915-956.
\bibitem{LorscheidWeist1} O.~Lorscheid, T.~Weist, \emph{Representation type vie Euler characteristic and singularities of quiver Grassmannians}. Bull. LMS \textbf{51} (2019), no. 5, 815-835.
\bibitem{LorscheidWeist2} O.~Lorscheid, T.~Weist, \emph{Pl\"ucker relations for quiver Grassmannians}. Algebr. Represent. Theory \textbf{22} (2019), no. 1, 211-218.
\bibitem{LorscheidWeist3} O.~Lorscheid, T.~Weist. \emph{Quiver Grassmannians of extended Dynkin type $\tilde{D}_n$- Part 2: Schubert decompositions and F-polynomials}. Algebr. Represent. Theory \textbf{26} (2023), no. 2, 359-409.
\bibitem{LorscheidWeist4} O.~Lorscheid, T.~Weist. \emph{Quiver Grassmannians of extended Dynkin type D- Part 1: Schubert systems and decompositions into affine spaces}. Mem. Amer. Math. Soc. \textbf{261} (2019), no. 1258.
\bibitem{McGertyNevins}
K.~McGerty, T.~Nevins. \emph{Kirwan surjectivity for quiver varieties}. Inv. Math., \textbf{212} (2018). 161-187.
\bibitem{ReEvery} M.~Reieneke. \emph{Every projective variety is a quiver Grassmannian}. Algebr. Represent. Theory \textbf{16} (2013), no. 5, 1313-1314.
\bibitem{Riedtmann}
C.~Riedtmann. \emph{Degenerations for representations of quivers with relations}. Ann. Scient. Ec. Norm. Sup, \textbf{19} (1986), 275-301.
\bibitem{Ringel} C.~M.~Ringel. \emph{Tame algebras and integral quadratic forms}. Lecture Notes in Mathematics. \textbf{1099} (1980). Springer.

\bibitem{RupelWeist} D.~Rupel, T.~Weist. \emph{Cell decompositions for rank two quiver Grassmannians}.  Math. Z. \textbf{295} (2020), no. 3-4, 993-1038.

\bibitem{Schofield} A.~Schofield. \emph{General representations of quivers}. Proc. LMS. \textbf{65}  (1992), 46--64.
\bibitem{Schiffler}R.~Schiffler. \emph{Quiver Representations}. CMS Books in Mathematics (2014). Springer. 
\bibitem{Spalt75} N.~Spaltenstein. \emph{The fixed point set of a unipotent transformation on the flag manifold}.  Indag. Math. \textbf{38} (1976), no 5, 452-456. 
\bibitem{Tr} D.~Treumann. \emph{Exit paths and constructible stacks}. Comp. Math. \textbf{145} (2009), no. 6, 1504-1532.
\bibitem{Zwara} G.~Zwara. \emph{Degenerations for modules over representation-finite algebras}. Proc. AMS \textbf{127} (1999). no.5, 1313-1322.
\bibitem{ZwaraSpecialBiserial} G.~Zwara. \emph{Degenerations for modules over representation-finite biserial algebras}. J.~Algebra \textbf{198} (1997), 563-581.
\bibitem{Zwara-unibranch} G.~Zwara. \emph{Unibranch orbit closures in module varieties}. Ann. Scient. Ec. Norm. Sup. \textbf{35} (2002), 877-895.

\end{thebibliography}
 \end{document}